\documentclass[a4paper,11pt]{article}
\usepackage{makeidx}
\usepackage{graphicx}
\usepackage[english,activeacute]{babel}
\usepackage[latin1]{inputenc}
\usepackage{amsmath}
\usepackage{amsfonts}
\usepackage[hidelinks]{hyperref}
\usepackage{amssymb}
\usepackage{latexsym}
\usepackage{colortbl}
\usepackage{multicol}
\usepackage{hyperref}
\usepackage{xcolor}
\usepackage{mathrsfs}
\usepackage{orcidlink}

\newtheorem{theorem}{Theorem}

\newtheorem{corollary}{Corollary}

\newtheorem{lemma}{Lemma}

\newtheorem{proposition}{Proposition}
\newtheorem{remark}{Remark}

\newenvironment{proof}[1][Proof]{\noindent\textbf{#1.} }{\ \rule{0.5em}{0.5em}}
\begin{document}

\title{On zero behavior of higher-order Sobolev-type discrete q-Hermite I orthogonal polynomials}
\author{ {Edmundo J. Huertas$^{1}$\orcidlink{0000-0001-6802-3303}}, {Alberto Lastra$^{1}$\orcidlink{0000-0002-4012-6471}}, {Anier Soria-Lorente$^{2}$\orcidlink{0000-0003-3488-3094}}, {V\'{i}ctor Soto-Larrosa$^{1,\dagger }$\orcidlink{0000-0002-7079-3646}} \\
\\
$^{1}$Departamento de F\'isica y Matem\'aticas, Universidad de Alcal\'a\\
Ctra. Madrid-Barcelona, Km. 33,600\\
28805 - Alcal\'a de Henares, Madrid, Spain\\
edmundo.huertas@uah.es, alberto.lastra@uah.es, v.soto@uah.es\\
\\
$^{2}$Departamento de Tecnolog\'ia, Universidad de Granma\\
Ctra. de Bayamo-Manzanillo, Km. 17,500\\
85100 - Bayamo, Cuba\\
asorial@udg.co.cu, asorial1983@gmail.com}
\maketitle

\begin{abstract}
In this work, we investigate the sequence of monic q-Hermite I-Sobolev type orthogonal polynomials of higher-order, denoted as $\{\mathbb{H}%
_{n}(x;q)\}_{n\geq 0}$, which are orthogonal with respect to the following non-standard inner product involving q-differences:

\begin{equation*}
\langle p,q\rangle _{\lambda }=\int_{-1}^{1}f\left( x\right) g\left(
x\right) (qx,-qx;q)_{\infty }d_{q}(x)+\lambda \,(\mathscr{D}%
_{q}^{j}f)(\alpha)(\mathscr{D}_{q}^{j}g)(\alpha),
\end{equation*}

where $\alpha \in \mathbb{R}\backslash (-1,1)$, $\lambda $ belongs to the
set of positive real numbers, $\mathscr{D}_{q}^{j}$ denotes the $j$-th $q $%
-discrete analogue of the derivative operator, and $(qx,-qx;q)_{\infty
}d_{q}(x)$ denotes the orthogonality weight with its points of increase in a
geometric progression. We proceed to obtain the hypergeometric
representation of $\mathbb{H}_{n}(x;q)$ and explicit expressions for the
corresponding ladder operators. From the latter, we obtain a novel kind of
three-term recurrence formula with rational coefficients associated with
these polynomial family. Moreover, for certain real values of $\alpha $, we
present some results concerning the location of the zeros of $\mathbb{H}%
_n(x;q)$ and we perform a comprehensive analysis of their asymptotic
behavior as the parameter $\lambda $ varies from zero to infinity.

\vspace{0.3cm}



{\footnotesize $\dagger$ Corresponding author.}
\end{abstract}



\section{Introduction}

\label{S1-Intro}



The \textit{discrete }$q$\textit{-Hermite I polynomials} of degree $n$,
usually denoted in the literature as $H_{n}(x;q)$, are a family of $q$%
-hypergeometric polynomials introduced in 1965 by W. A. Al--Salam and L.
Carlitz in \cite{AC1965}. They are the Al-Salam-Carlitz I polynomials with
parameter $a=-1$, and they are usually defined by means of their generating
function (see \cite[Sec. 14.28]{KLS2010})%
\begin{equation}
\sum\limits_{n=0}^{\infty }\frac{H_{n}(x;q)}{(q;q)_{n}}t^{n}=\frac{%
(t^{2};q^{2})_{\infty }}{(xt;q)_{\infty }}{,}  \label{e1}
\end{equation}%
where $q$ stands for their unique parameter, for which we assume that $0<q<1$%
, which means that they belong to the class of orthogonal polynomial
solutions of certain second order $q$-difference equations, known in the
literature as the Hahn class (see \cite{H1949}, \cite{KLS2010}). This
sequence of polynomials lies at the bottom of the Askey-scheme of
hypergeometric orthogonal polynomials, and for $q=1$ one recovers the
classical Hermite polynomials while for $q=0$ one recovers the re-scaled
Chebyshev polynomials of the second kind. It is well known that the
orthogonal polynomial family $\{H_{n}(x;q)\}_{n=0}^{\infty }$ finds
widespread applications in various theoretical and mathematical physics
scenarios. These applications extend to areas such as quantum groups (see,
for example \cite{K-OPTP90}, \cite{K-LST94}), discrete mathematics,
algebraic combinatorics (\cite{B-OPTP90}), as well as in the study of q-Schrö%
dinger equations and q-harmonic oscillators (\cite{AAS-SS93}, \cite{AS-LMP93}%
, \cite{AS-TMP85}, \cite{AS-TMP87} \cite{M-JPAMG89}).

\smallskip

On the other hand, the so called \textit{Sobolev orthogonal polynomials}
refer to those families of polynomials orthogonal with respect to inner
products involving positive Borel measures supported on infinite subsets of
the real line, and also involving regular derivatives. When these
derivatives appear only on function evaluations on a finite discrete set,
the corresponding families are called \textit{Sobolev-type} or \textit{%
discrete Sobolev} orthogonal polynomial sequences. For a recent and
comprehensive survey on the subject, we refer to \cite{MX2015} and the
references therein. Some of the techniques used in the paper have been taken
from the classical references \cite{LM1995} and \cite{MR1990}, but here
applied to $q$-lattices.

\smallskip

In the present work, our focus centers on a Sobolev-type family of
orthogonal polynomials associated with the $q$-Hermite I polynomials.
Specifically, we delve into the study of a higher-order sequence of $q$%
-Hermite I-Sobolev type polynomials, denoted as ${\mathbb{H}_{n}(x;q)}%
_{n\geq 0}$, which exhibit orthogonality concerning the Sobolev-type inner
product%
\begin{equation*}
\left\langle f,g\right\rangle _{\lambda }=\langle f,g\rangle +\lambda ({%
\mathscr D}_{q}^{j}f)(\alpha )({\mathscr D}_{q}^{j}g)(\alpha ),
\end{equation*}%
where $\left\langle \cdot ,\cdot \right\rangle $ stands for the inner
product associated to the orthogonality condition (see equation (\ref%
{eq:innerprod})) satisfied by the discrete $q-$Hermite I polynomials, $%
\lambda $ stands for a fixed positive real number, $\alpha $ is a real
number outside the support of the orthogonality measure related to $\langle
\cdot ,\cdot \rangle $, $j\in \mathbb{N}$, and ${\mathscr D}_{q}^{j}$ stands
for the $j$-th iterate of the Euler-Jackson $q$-difference operator (see
equation (\ref{eulerjackson})). To the best of our knowledge, when $\alpha $
is an arbitrary real number such that $|\alpha |>1$, the study of these
higher-order $q$-discrete Sobolev polynomials has not been carried out in
the literature. Thus, the major reason for this study is to bring together
several algebraic and analytic properties of the polynomials $\{\mathbb{H}%
_{n}(x;q)\}_{n\geq 0}$ as well as conducting a comprehensive analysis of the
zeros of the sequence, providing numerical evidence of them.

\smallskip

The paper is structured as follows. In Section~\ref{sec2} we recall the
basic facts and definitions from $q$-calculus and some of the main elements
associated to the $q$-Hermite I polynomials. In Section~\ref{S3-ConnForm},
we define the $q$-Hermite I Sobolev-type polynomials of higher order and
state some properties relating them to the $q$-Hermite I polynomials via
connection formulas, and also providing their basic hypergeometric
representation. In Section~\ref{secpral} we state two structure relations
and the expressions for the ladder operators associated to $\mathbb{H}%
_n(x;q) $, from which we obtain a three term recurrence relation with
non-constant coefficients (Theorem~\ref{S4-Theor3TRR-RC}). In section \ref%
{sec:zeros}, we study the location of the zeros of $\mathbb{H}_n(x;q)$ with
respect to the interval $(-1,1)$ and with respect to the mass point $\alpha$%
. We also deduce interlacing properties between the zeros of the Sobolev and
the standard sequences, as well as we investigate the asymptotic behaviour
of the zeros of $\mathbb{H}_n(x;q)$ as a function of $\lambda$ when $\lambda$
tends from zero to infinity. Some numerical examples support our results in
this section. The paper concludes with a section of conclusions, further
remarks and examples. 


\section{Preliminaries}

\label{sec2}



This preliminary section is devoted to recall the main results from $q$%
-calculus we will use through this paper, as well as the definition and main
properties of the discrete $q$-Hermite I polynomials.



\subsection{Results from $q$-calculus}



We begin by providing some background and definitions from $q$-calculus. We
refer to \cite{KLS2010} for further details. Given $n\in \mathbb{N}$, the $q$%
-number $[n]_{q}$, is denoted by%
\begin{equation*}
\lbrack n]_{q}=%
\begin{cases}
\displaystyle0, & \mbox{ if }n=0, \\ 
&  \\ 
\displaystyle\frac{1-q^{n}}{1-q}=\sum_{0\leq k\leq n-1}q^{k}, & \mbox{ if }%
n\geq 1.%
\end{cases}%
\end{equation*}%
It makes sense to extend the previous definition to $n\in \mathbb{Z}$ by $%
[n]_{q}=(1-q^{n})/(1-q)$. From the previous definition, a $q$-analogue of
the factorial of $n$ can be stated by 
\begin{equation*}
\lbrack n]_{q}!=%
\begin{cases}
\displaystyle1, & \mbox{if }n=0, \\ 
&  \\ 
\displaystyle\lbrack n]_{q}[n-1]_{q}\cdots \lbrack 2]_{q}[1]_{q}, & 
\mbox{if
}n\geq 1.%
\end{cases}%
\end{equation*}%
In addition to this, we will make use of a $q$-analogue of the Pochhammer
symbol, or shifted factorial, which is given by 
\begin{equation*}
\left( a;q\right) _{n}=%
\begin{cases}
\displaystyle1, & \mbox{if }n=0, \\ 
&  \\ 
\displaystyle\prod_{0\leq j\leq n-1}\left( 1-aq^{j}\right) , & \mbox{if }%
n\geq 1, \\ 
&  \\ 
\displaystyle(a;q)_{\infty }=\prod_{j\geq 0}(1-aq^{j}), & \mbox{if }n=\infty %
\mbox{ and }|a|<1.%
\end{cases}%
\end{equation*}%
Let us fix the following compact notation%
\begin{equation*}
(a_{1},\ldots ,a_{r};q)_{k}=\prod_{j=1}^{r}(a_{j};q)_{k}.
\end{equation*}

Given two finite sequences of complex numbers $\left\{ a_{i}\right\}
_{i=1}^{r}$ and $\left\{ b_{i}\right\} _{i=1}^{s}$ under the assumption that 
$b_{j}\neq q^{-n}$ with $n\in \mathbb{N}$, for $j=1,2,\ldots ,s$, the basic
hypergeometric, or $q$-hypergeometric, $_{r}\phi _{s}$ series with variable $%
z$ is defined by 
\begin{equation*}
_{r}\phi _{s}\left( 
\begin{array}{c}
a_{1},a_{2},\ldots ,a_{r} \\ 
b_{1},b_{2},\ldots ,b_{s}%
\end{array}%
;q,z\right) =\sum_{k=0}^{\infty }\frac{\left( a_{1},\ldots ,a_{r};q\right)
_{k}}{\left( b_{1},\ldots ,b_{s};q\right) _{k}}\left( (-1)^{k}q^{\binom{k}{2}%
}\right) ^{1+s-r}\frac{z^{k}}{\left( q;q\right) _{k}}.
\end{equation*}

We follow \cite{arvesu2013first} in defining the $q$-falling factorial by%
\begin{equation*}
\left[ s\right] _{q}^{(0)}=1,\quad \quad \left[ s\right] _{q}^{(n)}=\frac{%
(q^{-s};q)_{n}}{(q-1)^{n}}q^{ns-\binom{n}{2}},\,\text{for}\,n\geq 1,\,\text{%
and}\,\binom{n}{2}=0,\,\text{if}\,n<2,
\end{equation*}%
noticing that $[s]_{q}^{(1)}$ coincides with the $q$-number $[s]_{q}$. From
the previous definition, one can state the $q$-analog of the binomial
coefficient (see \cite[(1.9.4)]{KLS2010}), which is given by%
\begin{equation*}
\begin{bmatrix}
n \\ 
k%
\end{bmatrix}%
_{q}=\frac{\left( q;q\right) _{n}}{\left( q;q\right) _{k}\left( q;q\right)
_{n-k}}=\frac{\left[ n\right] _{q}!}{\left[ k\right] _{q}!\left[ n-k\right]
_{q}!},\quad k=0,1,\ldots ,n,
\end{equation*}%
where $n$ denotes a nonnegative integer.

Following \cite{ernst2007q}, the Jackson-Hahn-Cigler $q$-subtraction
operator is defined by%
\begin{equation*}
\left( x\boxminus _{q}y\right) ^{n}=\sum_{k=0}^{n}%
\begin{bmatrix}
n \\ 
k%
\end{bmatrix}%
_{q}q^{\binom{k}{2}}(-y)^{k}x^{n-k}.
\end{equation*}

It will be useful for determining a more compact writing for the derivatives
of the reproducing kernel of the sequence of $q$-Hermite I polynomials, and
all the elements determined from it. At this point, we define the $q$%
-derivative, or the Euler--Jackson $q$-difference operator, by (see \cite%
{KLS2010})%
\begin{equation}
({\mathscr D}_{q}f)(z)=%
\begin{cases}
\displaystyle\frac{f(qz)-f(z)}{(q-1)z}, & \text{if}\ z\neq 0,\ q\neq 1, \\ 
&  \\ 
f^{\prime }(z), & \text{if}\ z=0,\ q=1,%
\end{cases}
\label{eulerjackson}
\end{equation}%
where ${\mathscr D}_{q}^{0}f=f$, ${\mathscr D}_{q}^{n}f={\mathscr D}_{q}({%
\mathscr D}_{q}^{n-1}f)$, with $n\geq 1$. We observe that 
\begin{equation*}
\lim\limits_{q\rightarrow 1}{\mathscr D}_{q}f(z)=f^{\prime }(z).
\end{equation*}%
The previous $q$-analog of the derivative operator will determine two
functional equations satisfied by the Sobolev type polynomials defined in
the present work. Moreover, the $q$-derivative satisfies the following
algebraic statements%
\begin{equation}
{\mathscr D}_{q}[f(\gamma z)]=\gamma ({\mathscr D}_{q}f)(\gamma z),\quad
\forall \,\,\gamma \in \mathbb{C},  \label{cadrule}
\end{equation}%
\begin{equation}
{\mathscr D}_{q}f(z)={\mathscr D}_{q^{-1}}f(qz),  \label{DqProp}
\end{equation}%
\begin{equation}
{\mathscr D}_{q}[f(z)g(z)]=f(qz){\mathscr D}_{q}g(z)+g(z){\mathscr D}%
_{q}f(z),  \label{prodqD}
\end{equation}%
\begin{equation}
{\mathscr D}_{q}({\mathscr D}_{q^{-1}}f)(z)=q^{-1}{\mathscr D}_{q^{-1}}({%
\mathscr D}_{q}f)(z).  \label{DqProOK}
\end{equation}%
also a q-analog of the Leibniz' rule can be defined%
\begin{equation}
{\mathscr D}_{q}^{n}[f(z)g(z)]=\sum_{k=0}^{n}%
\begin{bmatrix}
n \\ 
k%
\end{bmatrix}%
_{q}(\mathscr{D}_{q}^{k}f)(z)\cdot (\mathscr{D}_{q}^{n-k}g)(q^{k}z),\quad
n=0,1,2,\ldots .  \label{eq:Leibniz}
\end{equation}%
all this definitions allow us to introduce the $q-$Taylor series for a
function $f(x)$ centered at $x_{0}$%
\begin{equation*}
f(x)=\displaystyle\sum_{k=0}^{\infty }\frac{\mathscr{D}_{q}^{k}f(x)}{[k]_{q}!%
}(x\boxminus _{q}x_{0})^{k}
\end{equation*}%
finally we are making use of the Jackson $q-$integral (see, for example \cite%
[Sec. 11.4]{Ism2005})

\begin{equation*}
\displaystyle\int_{a}^{b}f(t)d_{q}t=\displaystyle\int_{a}^{0}f(t)d_{q}t+%
\displaystyle\int_{0}^{b}f(t)d_{q}t,
\end{equation*}%
for $a<0<b$ where%
\begin{equation*}
\displaystyle\int_{a}^{0}f(t)d_{q}t=-a(1-q)\displaystyle\sum_{n=0}^{\infty
}f(q^{n}a)q^{n}
\end{equation*}%
and%
\begin{equation*}
\displaystyle\int_{0}^{b}f(x)d_{q}t=b(1-q)\displaystyle\sum_{n=0}^{\infty
}f(q^{n}b)q^{n}.
\end{equation*}%
%
%
%
%
%
%


\subsection{Discrete $q$-Hermite I orthogonal polynomials}

\label{S11-Prelim}



After the above section recalling the main elements of $q$-calculus, we
continue by giving several aspects and properties of the discrete $q$%
-Hermite I polynomials $\{H_{n}(x;q)\}_{n\geq 0}$. Departing from the
generating function (\ref{e1}), one has that the monic discrete $q$-Hermite
I polynomials can be explicitly given by (see \cite[f. (14.24.1) with $a=-1$]%
{KLS2010})%
\begin{equation}
H_{n}(x;q)=q^{\binom{n}{2}}\,_{2}\phi _{1}\left( 
\begin{array}{c}
q^{-n},x^{-1} \\ 
0%
\end{array}%
;q,-qx\right) ,  \label{ASP}
\end{equation}%
satisfying the orthogonality relation (see \cite{KLS2010})%
\begin{equation}
\int_{-1}^{1}H_{m}(x;q)H_{n}(x;q)(qx,-qx;q)_{\infty
}d_{q}x=(1-q)(q;q)_{n}(q,-1,-q;q)_{\infty }q^{\binom{n}{2}}\delta _{m,n},
\label{eq:innerprod}
\end{equation}%
where $\delta _{m,n}$ stands for Kronecker delta.


The following statements can be derived from the previous definitions (see~%
\cite{KLS2010} for further details).

\begin{proposition}
\label{S1-Proposition11} Let $\{H_{n}(x;q)\}_{n\geq 0}$ be the sequence of $%
q $-Hermite I polynomials of degree $n$. Then, the following statements hold.

\begin{enumerate}
\item Recurrence relation. The recurrence relation holds for every integer $%
n\ge0$ 
\begin{equation}
xH_{n}(x;q) =H_{n+1}(x;q)+\gamma_n H_{n-1}(x;q),  \label{ReR}
\end{equation}
with initial conditions $H_{-1}(x;q)=0$ and $H_{0}(x;q)=1$. Here, $%
\gamma_n=q^{n-1}(1-q^{n})$.

\item Squared norm. For every $n\in \mathbb{N}$, 
\begin{equation*}
||H_{n}||^{2}=(1-q)(q;q)_n(q,-1,-q;q)_{\infty}q^{\binom{n}{2}}.
\end{equation*}

\item Forward shift operator. For every $n,k\in \mathbb{N}$,%
\begin{equation}
{\mathscr D}_{q}^{k}H_{n}(x;q)=[n]_{q}^{(k)}H_{n-k}(x;q),\text{ with }%
H_{n-k}(x;q)=0\text{ if }n<k,  \label{FSop}
\end{equation}%
where we recall that%
\begin{equation*}
\left[ n\right] _{q}^{(k)}=\frac{(q^{-n};q)_{k}}{(q-1)^{k}}q^{kn-\binom{k}{2}%
},\text{ with }\binom{k}{2}=0\text{ if }k<2,
\end{equation*}%
denotes the $q$-falling factorial.

\item Second-order $q$-difference equation.%
\begin{equation*}
\sigma (x){\mathscr D}_{q}{\mathscr D}_{q^{-1}}H_{n}(x;q)+\tau (x){\mathscr D%
}_{q}H_{n}(x;q)+\lambda _{n,q}H_{n}(x;q)=0,
\end{equation*}%
where $\sigma (x)=x^{2}-1$, $\tau (x)=(1-q)^{-1}x$ and $\lambda
_{n,q}=[n]_{q}([1-n]_{q}\sigma ^{\prime \prime }/2-\tau ^{\prime })$.
\end{enumerate}
\end{proposition}

It is worth mentioning that the previous $q$-difference equation appears in 
\cite[eq. (14.28.5)]{KLS2010} in an equivalent form.



\begin{proposition}[Christoffel-Darboux formula]
\label{S1-Proposition12} Let $\{H_{n}(x;q)\}_{n\geq 0}$ be the sequence of $q
$-Hermite I polynomials. If we denote the $n$-th reproducing kernel by%
\begin{equation*}
K_{n,q}(x,y)=\sum_{k=0}^{n}\frac{H_{k}(x;q)H_{k}(y;q)}{||H_{k}||^{2}}.
\end{equation*}%
Then, for all $n\in \mathbb{N}$, it holds that%
\begin{equation}
K_{n,q}(x,y)=\frac{H_{n+1}(x;q)H_{n}(y;q)-H_{n+1}(y;q)H_{n}(x;q)}{\left(
x-y\right) ||H_{n}||^{2}}  \label{CDarb}
\end{equation}%
with $x\neq y$.
\end{proposition}



Let us fix the following notation for the partial $q$-derivatives of $%
K_{n,q}(x,y)$%
\begin{eqnarray}
K_{n,q}^{(i,j)}(x,y) &=&{\mathscr D}_{q,y}^{j}({\mathscr D}%
_{q,x}^{i}K_{n,q}(x,y))  \notag \\
&=&\sum_{k=0}^{n}\frac{{\mathscr D}_{q}^{i}H_{k}(x;q){\mathscr D}%
_{q}^{j}H_{k}(y;q)}{||H_{k}||^{2}}.  \label{eq:KernelDij}
\end{eqnarray}%
Christoffel-Darboux formula allow us to give a more precise writing for the
derivatives of the $n$-th reproducing kernel associated to the sequence of $q
$-Hermite I polynomials in the next result. Its proof follows an analogous
technique as that \cite[Lemma 1]{hermoso2020second}, so therefore, we omit
the proof here.



\begin{lemma}
\label{S1-LemmaKernel0j} Let $\{H_{n}(x;q)\}_{n\geq 0}$ be the sequence of $%
q $-Hermite I polynomials of degree $n$. Then, the following statement holds
for all $n\in \mathbb{N}$ 
\begin{equation}
K_{n-1,q}^{(0,j)}(x,y)={\mathcal{A}}(x,y)H_{n}(x;q)+{\mathcal{B}}%
(x,y)H_{n-1}(x;q),  \label{Kernel0j}
\end{equation}%
where $\mathcal{A}(x,y) = \mathcal{A}(x,y;n,j,q) $ and $\mathcal{B}(x,y) = 
\mathcal{B}(x,y;n,j,q) $ are rational functions given by 
\begin{equation*}
{\mathcal{A}}(x,y)=\frac{\left[ j\right] _{q}!}{||H_{n-1}||^{2}\left(
x\boxminus_q y\right) ^{j+1}}\sum_{k=0}^{j}\frac{{\mathscr D}%
_{q}^{k}H_{n-1}(y;q)}{\left[ k\right] _{q}!}(x\boxminus_q y)^{k},
\end{equation*}%
and 
\begin{equation*}
{\mathcal{B}}(x,y)=-\frac{\left[ j\right] _{q}!}{||H_{n-1}||^{2}\left(
x\boxminus_q y\right) ^{j+1}}\sum_{k=0}^{j}\frac{{\mathscr D}%
_{q}^{k}H_{n}(y;q)}{\left[ k\right] _{q}!}(x\boxminus_q y)^{k},
\end{equation*}%
with $H_{-1}(x;q)=0$, and $K_{-1,q}(x,y)=0$.
\end{lemma}

\begin{remark}
Notice that%
\begin{equation*}
K_{n-1,q}(x,\alpha )=\frac{[j]_{q}!}{||H_{n-1}||^{2}(x\boxminus _{q}\alpha
)^{j+1}}\left( H_{n}(x;q)\cdot Q_{j,q}(x,\alpha ,H_{n-1})-H_{n-1}(x;q)\cdot
Q_{j,q}(x,\alpha ,H_{n})\right) 
\end{equation*}%
where $Q_{j,q}(x,\alpha ,H_{n-1})$ and $Q_{j,q}(x,\alpha ,H_{n})$ are the $q-
$Taylor polynomials of degree $j$ of the polynomials $H_{n-1}(y;q)$ and $%
H_{n}(y;q)$ at $y=\alpha $, respectively.
\end{remark}



\begin{lemma}
\label{S1-LemmaKerneli2} Let $\{H_{n}(x;q)\}_{n\geq 0}$ be the sequence of $q
$-Hermite I polynomials of degree $n$. Then, the following statements hold,
for all $n,j\in \mathbb{N}$,%
\begin{equation}
K_{n-1,q}^{(1,j)}(x,y)=\mathcal{C}(x,y)H_{n}(x;q)+\mathcal{D}%
(x,y)H_{n-1}(x;q),  \label{kernel1j}
\end{equation}%
with $\mathcal{C}(x,y)=\mathcal{C}(x,y;n,j,q)$ and $\mathcal{D}(x,y)=%
\mathcal{D}(x,y;n,j,q)$ given by 
\begin{equation*}
\mathcal{C}(x,y)={\mathscr D}_{q}{\mathcal{A}}(x,y)-[n-1]_{q}\gamma
_{n-1}^{-1}\mathcal{B}(x,y)(qx,y),
\end{equation*}%
and 
\begin{equation*}
\mathcal{D}(x,y)=[n]_{q}\mathcal{A}(x,y)(qx,y)+[n-1]_{q}\gamma _{n-1}^{-1}x%
\mathcal{B}(x,y)(qx,y)+{\mathscr D}_{q}\mathcal{B}(x,y),
\end{equation*}%
%
%
%
%
%
%
%
%
%
%
%
%
%
%
%
with $H_{-1}(x;q)=0$, and $K_{-1,q}(x,y)=0$.
\end{lemma}

\begin{proof}
Applying the $q$-derivative operator ${\mathscr D}_{q}$ with respect to $x$
variable to (\ref{Kernel0j}), together with the property (\ref{prodqD})
yields%
\begin{equation*}
K_{n-1,q}^{(1,j)}(x,y)=\mathcal{A}(x,y)(qx){\mathscr D}%
_{q}H_{n}(x;q)+H_{n}(x;q){\mathscr D}_{q}\mathcal{A}(x,y)(x)
\end{equation*}%
\begin{equation*}
+\mathcal{B}(x,y)(qx){\mathscr D}_{q}H_{n-1}(x;q)+H_{n-1}(x;q){\mathscr D}%
_{q}{\mathcal{B}(x,y)(x)}.
\end{equation*}%
Using (\ref{prodqD}), (\ref{ReR}) and (\ref{FSop}) we deduce (\ref{kernel1j}%
). 
\end{proof}



\section{Connection formulas and hypergeometric representation}

\label{S3-ConnForm}



In this section we introduce the $q$-Hermite I-Sobolev type polynomials of
higher order $\{\mathbb{H}_{n}(x;q)\}_{n\geq 0}$, which are orthogonal with
respect to the Sobolev-type inner product%
\begin{equation}
\left\langle f,g\right\rangle _{\lambda
}=\int_{-1}^{1}f(x)g(x)(qx,-qx;q)_{\infty }d_{q}x+\lambda ({\mathscr D}%
_{q}^{j}f)(\alpha ;q)({\mathscr D}_{q}^{j}g)(\alpha ;q),  \label{piSob}
\end{equation}%
where $\alpha \in \mathbb{R}\setminus (-1,1)$, $\lambda >0$ and $j\geq 1$.
In addition, we express the $q$-Hermite I-Sobolev type polynomials of higher
order $\{\mathbb{H}_{n}(x;q)\}_{n\geq 0}$ in terms of the $q$-Hermite I
polynomials $\{H_{n}(x;q)\}_{n\geq 0}$, their associated kernel polynomials
and their corresponding $q-$derivatives. Moreover, we obtain the basic
hypergeometric representation of the proposed polynomials.



\begin{proposition}
Let $\{\mathbb{H}_{n}(x;q)\}_{n\geq 0}$ be the sequence of $q$-Hermite
I-Sobolev type polynomials of degree $n$. Then, the following statements
hold for $n\geq 1$ 
\begin{equation}
\mathbb{H}_{n}(x;q)=H_{n}(x;q)-\lambda \frac{\lbrack
n]_{q}^{(j)}H_{n-j}(\alpha ;q)}{1+\lambda K_{n-1,q}^{(j,j)}(\alpha ,\alpha )}%
K_{n-1,q}^{(0,j)}(x,\alpha ).  \label{ConxF1}
\end{equation}%
for $n\geq 0$, with $H_{n-k}(x;q)=0$ if $n<k$, and $K_{-1,q}(x,y)=0$.
\end{proposition}

\begin{proof}
From section 2 of \cite{MR1990} we have 
\begin{equation}
\mathbb{H}_{n}(x;q)=H_{n}(x;q)-\lambda \frac{{\mathscr D}^{j}_{q}H_{n}(%
\alpha,q)}{ 1+\lambda K^{(j,j)}_{n-1,q}(\alpha,\alpha) }
K^{(0,j)}_{n-1,q}(x,\alpha),  \label{eq:DqconexionI}
\end{equation}
and using \ref{FSop} the result holds.
\end{proof}

As a consequence, we have the following results.



\begin{corollary}
\label{DqDq2HS} Let $\{\mathbb{H}_{n}(x;q)\}_{n\geq 0}$ be the sequence of $q$-Hermite I-Sobolev type polynomials of degree $n$. Then, the following statements hold for $n\geq 0$, 
\begin{equation*}
{\mathscr D}_{q}\mathbb{H}_{n}(x;q)=[n]_{q}H_{n-1}(x;q)-\lambda \frac{%
\lbrack n]_{q}^{(j)}H_{n-j}(\alpha ;q)}{1+\lambda K_{n-1,q}^{(j,j)}(\alpha
,\alpha )}K_{n-1,q}^{(1,j)}(x,\alpha ),
\end{equation*}%
with $H_{n-k}(x;q)=0$ if $n<k$, and $K_{-1,q}(x,y)=0$.

\end{corollary}



\begin{lemma}
Let $\{\mathbb{H}_{n}(x;q)\}_{n\geq 0}$ be the sequence of $q$-Hermite
I-Sobolev type polynomials of degree $n$. Then, the following statements
hold for $n\geq 0$, 
\begin{equation}
\mathbb{H}_{n}(x;q)={\mathcal{E}}_{1,n}(x)H_{n}(x;q)+{\mathcal{F}}%
_{1,n}(x)H_{n-1}(x;q),  \label{ConexF_I}
\end{equation}%
where 
\begin{equation}
{\mathcal{E}}_{1,n}(x)=1-\lambda \frac{\lbrack n]_{q}^{(j)}H_{n-j}(\alpha ;q)%
}{1+\lambda K_{n-1,q}^{(j,j)}(\alpha ,\alpha )}{\mathcal{A}}(x,\alpha ),
\label{EEq}
\end{equation}%
and 
\begin{equation}
{\mathcal{F}}_{1,n}(x)=-\lambda \frac{\lbrack n]_{q}^{(j)}H_{n-j}(\alpha ;q)%
}{1+\lambda K_{n-1,q}^{(j,j)}(\alpha ,\alpha )}{\mathcal{B}}(x,\alpha ),
\label{FEq}
\end{equation}%
with $\mathbb{H}_{1}(x;q)=0$, and $H_{-1}(x;q)=0$.
\end{lemma}



\begin{proof}
From (\ref{ConxF1}) and Proposition \ref{S1-LemmaKernel0j} we conclude the
result.
\end{proof}



On the other hand, from the previous Lemma and recurrence relation (\ref{ReR}%
) we have the following result for $n\geq 0$%
\begin{equation}
\mathbb{H}_{n-1}(x;q)={\mathcal{E}}_{2,n}(x)H_{n}(x;q)+{\mathcal{F}}%
_{2,n}(x)H_{n-1}(x;q),  \label{ConexF_II}
\end{equation}%
where 
\begin{equation*}
{\mathcal{E}}_{2,n}(x)=-\frac{{\mathcal{F}}_{1,n-1}(x)}{\gamma _{n-1}},
\end{equation*}%
and 
\begin{equation*}
{\mathcal{F}}_{2,n}(x)={\mathcal{E}}_{1,n-1}(x)-x{\mathcal{E}}_{2,n}(x),
\end{equation*}%
with $\mathbb{H}_{-1}(x;q)=0$, and $H_{-1}(x;q)=0$.



\begin{lemma}
\label{detxi} Let $\{\mathbb{H}_{n}(x;q)\}_{n\geq 0}$ be the sequence of $q$%
-Hermite I-Sobolev type polynomials of degree $n$. Then, the following
statements hold for $n\geq 0$%
\begin{equation}
\Theta_{1,2,n}(x)H_{n}(x;q)=%
\begin{vmatrix}
\mathbb{H}_{n}(x;q) & \mathbb{H}_{n-1}(x;q) \\ 
{\mathcal{F}}_{1,n}(x) & {\mathcal{F}}_{2,n}(x)%
\end{vmatrix}%
,  \label{ConexF_III}
\end{equation}%
and 
\begin{equation}
\Theta _{1,2,n}(x)H_{n-1}(x;q)=-%
\begin{vmatrix}
\mathbb{H}_{n}(x;q) & \mathbb{H}_{n-1}(x;q) \\ 
{\mathcal{E}}_{1,n}(x) & {\mathcal{E}}_{2,n}(x)%
\end{vmatrix}%
,  \label{ConexF_IV}
\end{equation}%
where we introduce the function $\Theta_{i,j,n}(x)$ for later convenience as follows 
\begin{equation*}
\Theta _{i,j,n}(x)=%
\begin{vmatrix}
{\mathcal{E}}_{i,n}(x) & {\mathcal{E}}_{j,n}(x) \\ 
{\mathcal{F}}_{i,n}(x) & {\mathcal{F}}_{j,n}(x)%
\end{vmatrix}%
,
\end{equation*}%
for $i,j = 1,2,3,4$ and with $\mathbb{H}_{-1}(x;q)=0$.
\end{lemma}



\begin{proof}
Let us multiply (\ref{ConexF_I}) by ${\mathcal{F}}_{2,n}(x)$ and (\ref%
{ConexF_II}) by $-{\mathcal{F}}_{1,n}(x)$. Adding and simplifying the
resulting equations, we deduce (\ref{ConexF_III}). In addition, we can
proceed analogously to get (\ref{ConexF_IV}).
\end{proof}

Finally, we will focus our attention on the hypergeometric representation of 
$\mathbb{H}_{n}(x;q)$. We outline the proof and omit the details of it, that
follows a similar analysis to that carried out in \cite%
{costas2018analytic,hermoso2020second}.



\begin{proposition}[Hypergeometric character]
\label{S3-Theorem31} The q-Hermite I-Sobolev type polynomials of higher
order $\{\mathbb{H}_{n}(x;q)\}_{n\geq 0}$, have the following hypergeometric
representation for all $n\geq 0$,%
\begin{equation}
\mathbb{H}_{n}(x;q)=-\frac{{\mathcal{F}}_{1,n}(x)(1-\psi _{n}(x)q^{-1})q^{%
\binom{n}{2}-n+2}}{[n]_{q}\psi _{n}(x)(1-q)}\times \,_{3}\phi _{2}\left( 
\begin{array}{c}
q^{-n},x^{-1},\psi _{n}(x) \\ 
0,\psi _{n}(x)q^{-1}%
\end{array}%
;q,-qx\right)  \label{ASPTSHR}
\end{equation}%
where $\psi _{n}(x)=((1-q)\vartheta _{n}(x)+1)^{-1}$ and 
\begin{equation*}
\vartheta _{n}(x)=-\frac{q^{n-2}[n]_{q}{\mathcal{E}}_{1,n}(x)}{{\mathcal{F}}%
_{1,n}(x)}-[n-1]_{q}.
\end{equation*}
\end{proposition}



\begin{proof}
For $n=0$, a~trivial verification shows that (\ref{ASPTSHR}) yields $\mathbb{%
H}_{0}^{(a)}(x;q)=1$. For~$n \geq 1$, combining (\ref{ASP}) with (\ref%
{ConexF_I}) and the relations%
\begin{equation*}
(q^{1-n};q)_{k}=-\frac{q}{[n]_{q}}([k-1]_{q}-[n-1]_{q})(q^{-n};q)_{k}
\end{equation*}
where we define 
\begin{equation*}
[-1]_{q}:=\frac{1-q^{-1}}{1-q}=-q^{-1},
\end{equation*}
and 
\begin{equation*}
(q^{-n};q)_{k}=0,\quad n<k,
\end{equation*}%
yields%
\begin{equation*}
\mathbb{H}_{n}^{(a)}(x;q)=-q^{\binom{n}{2}-n+2}\frac{{\mathcal{F}}_{1,n}(x)}{%
[n]_{q}}\sum_{k=0}^{n}([k-1]_{q}+\vartheta
_{n}(x))(q^{-n};q)_{k}(x^{-1};q)_{k}\frac{(a^{-1}qx)^{k}}{(q;q)_{k}}.
\end{equation*}
For more details see \cite{hermoso2020second}.

On the other hand, using 
\begin{equation*}
\psi _{n}(x)=((1-q)\vartheta _{n}(x)+1)^{-1}
\end{equation*}
and 
\begin{equation*}
\vartheta _{n}(x)=-\frac{q^{n-2}[n]_{q}{\mathcal{E}}_{1,n}(x)}{{\mathcal{F}}%
_{1,n}(x)}-[n-1]_{q},
\end{equation*}
after~some straightforward calculations, we get%
\begin{equation*}
[k-1]_{q}+\vartheta _{n}(x)=\frac{1-\psi _{n}(x)q^{-1}}{\psi _{n}(x)(1-q)}%
\frac{(\psi _{n}(x);q)_{k}}{(\psi _{n}(x)q^{-1};q)_{k}}.
\end{equation*}%
Therefore%
\begin{equation*}
\mathbb{H}_{n}^{(a)}(x;q)=-\frac{{\mathcal{F}}_{1,n}(x)(1-\psi
_{n}(x)q^{-1})q^{\binom{n}{2}-n+2}}{[n]_{q}\psi _{n}(x)(1-q)}\sum_{k=0}^{n}%
\frac{(q^{-n};q)_{k}(x^{-1};q)_{k}(\psi _{n}(x);q)_{k}}{(\psi
_{n}(x)q^{-1};q)_{k}}\frac{(a^{-1}qx)^{k}}{(q;q)_{k}}
\end{equation*}
which coincides with (\ref{ASPTSHR}) and completes the proof.
\end{proof}



\begin{lemma}
\label{DqHS} Let $\{\mathbb{H}_{n}(x;q)\}_{n\geq 0}$ be the sequence of $q$%
-Hermite I-Sobolev type polynomials of degree $n$. Then, the following
statements hold for $n\geq 0$, 
\begin{equation}
{\mathscr D}_{q}\mathbb{H}_{n}(x;q)={\mathcal{E}}_{3,n}(x)H_{n}(x;q)+{%
\mathcal{F}}_{3,n}(x)H_{n-1}(x;q),  \label{DqHsE3F3}
\end{equation}%
and 
\begin{equation}
{\mathscr D}_{q}\mathbb{H}_{n-1}(x;q)={\mathcal{E}}_{4,n}(x)H_{n}(x;q)+{%
\mathcal{F}}_{4,n}(x)H_{n-1}(x;q),  \label{DqHsE4F4}
\end{equation}%
where 
\begin{equation*}
{\mathcal{E}}_{3,n}(x)=-\lambda \frac{\lbrack n]_{q}^{(j)}H_{n-j}(\alpha ;q)%
}{1+\lambda K_{n-1,q}^{(j,j)}(\alpha ,\alpha )}\mathcal{C}(x,\alpha ),
\end{equation*}

\begin{equation*}
{\mathcal{F}}_{3,n}(x)=[n]_{q}-\lambda \frac{\lbrack
n]_{q}^{(j)}H_{n-j}(\alpha ;q)}{1+\lambda K_{n-1,q}^{(j,j)}(\alpha ,\alpha )}%
\mathcal{D}(x,\alpha ),
\end{equation*}

\begin{equation*}
{\mathcal{E}}_{4,n}(x)=-\frac{\mathcal{F}_{3,n-1}(x)}{\gamma_{n-1}},
\end{equation*}%
and 
\begin{equation*}
{\mathcal{F}}_{4,n}(x)=\mathcal{E}_{3,n-1}(x)+\frac{x\mathcal{F}_{3,n-1}(x)}{%
\gamma_{n-1}}
\end{equation*}

with $\mathbb{H}_{n-k}(x;q)=0$ if $n<k$, and $K_{-1,q}(x,y)=0$.
\end{lemma}



\begin{proof}
From Corollary \ref{DqDq2HS} and Proposition \ref{S1-LemmaKerneli2}, we
deduce (\ref{DqHsE3F3}), then shifting $n\rightarrow n-1$ in (\ref{DqHsE3F3}%
) and using (\ref{ReR}) we obtain (\ref{DqHsE4F4}).
\end{proof}

\section{Ladder operators and recurrence relation}

\label{secpral}



This section states a three-term recursion formula with non-constant
coefficients satisfied by the elements of the sequence $\left\{\mathbb{H}%
_n(x;q)\right\}_{n\ge 0}$, which is directly obtained by a proper
combination of the ladder operators. We begin this section by introducing an
useful lemma involving two structure relations satisfied by the sequence $%
\left\{\mathbb{H}_n(x;q)\right\}_{n\ge 0}$.



\begin{lemma}
The sequence of q-Hermite I Sobolev-type polynomials $\left\{\mathbb{H}%
_n(x;q)\right\}_{n\ge 0}$ satisfies the following structure relations for
all $n\ge 0$.

\begin{equation}
\Theta_{1,2,n}(x)\mathscr{D}_q \mathbb{H}_n(x;q) = \Theta_{3,2,n}(x)\mathbb{H%
}_n(x;q)+\Theta_{1,3,n}(x)\mathbb{H}_{n-1}(x;q)  \label{eq:structure1}
\end{equation}
and 
\begin{equation}
\Theta_{1,2,n}(x)\mathscr{D}_q \mathbb{H}_{n-1}(x;q) = \Theta_{4,2,n}(x)%
\mathbb{H}_n(x;q)+\Theta_{1,4,n}(x)\mathbb{H}_{n-1}(x;q)
\label{eq:structure2}
\end{equation}
\end{lemma}



\begin{proof}
Multiplying (\ref{DqHsE3F3}) by $\Theta_{1,2,n}(x)$ and using Lemma \ref%
{detxi} we obtain (\ref{eq:structure1}). In the same fashion, multiplying (%
\ref{DqHsE4F4}) by $\Theta_{1,2,n}(x)$ and using Lemma \ref{detxi} we deduce
(\ref{eq:structure2})
\end{proof}



rearranging terms in the above two equations, we arrive at a fully
equivalent result in terms of the ladder q-difference operators

\begin{theorem}
\label{ladderops} For every $n\ge 0$, $\lambda \in \mathbb{R}^+$, $\alpha
\in \mathbb{R}$ and $q\in(0,1)$, let $\hat{a}_n$ and $\hat{a}^\dagger_n$ be
the q-difference operators defined by

\begin{equation}
\hat{a}_n = \Theta_{1,2,n}(x)\mathscr{D}_q-\Theta_{3,2,n}(x)I_d
\label{eq:lowering}
\end{equation}

\begin{equation}
\hat{a}^\dagger_n = \Theta_{1,2,n}(x)\mathscr{D}_q-\Theta_{1,4,n}(x)I_d
\label{eq:rising}
\end{equation}
where $I_d$ is the identity operator. The previous q-difference operators
are lowering and rising operators related to the sequence of q-Hermite
I-Sobolev type polynomials $\left\{\mathbb{H}_n(x;q)\right\}_{n\ge 0}$,
satisfying 
\begin{equation}
\hat{a}_n\mathbb{H}_n(x;q) = \Theta_{1,3,n}(x)\mathbb{H}_{n-1}(x;q)
\label{eq:lowerHn}
\end{equation}
\begin{equation}
\hat{a}^\dagger_n\mathbb{H}_{n-1}(x;q) = \Theta_{4,2,n}(x)\mathbb{H}_n(x;q)
\label{eq:risingHn-1}
\end{equation}
\end{theorem}



we next provide a three-term recurrence formula with non-constant
coefficients satisfied by the sequence $\left\{\mathbb{H}_n(x;q)\right\}_{n%
\ge 0}$



\begin{theorem}
\label{S4-Theor3TRR-RC}Let $\{\mathbb{H}_{n}(x;q)\}_{n\geq 0}$ be the
sequence of $q$-Hermite I-Sobolev type polynomials of degree $n$. Then, $%
\mathbb{H}_{n}(x;q)$ satisfies the following three-term recurrence relation
for all $n\geq 0$, 
\begin{equation}
\alpha_n(x)\mathbb{H}_{n+1}(x;q)=\beta_{n}(x)\mathbb{H}_{n}(x;q)+\gamma_n(x) 
\mathbb{H}_{n-1}(x;q),  \label{eq:3TRR}
\end{equation}%
where 
\begin{equation*}
\alpha_n(x) = \Theta_{1,2,n}(x)\Theta_{4,2,n+1}(x)
\end{equation*}%
\begin{equation}
\beta_n(x) =
\Theta_{1,2,n+1}(x)\Theta_{3,2,n}(x)-\Theta_{1,2,n}(x)\Theta_{1,4,n+1}(x)
\label{eq:3TRRcoeffs}
\end{equation}%
and 
\begin{equation*}
\gamma_n(x) = \Theta_{1,2,n+1}(x)\Theta_{1,3,n}(x)
\end{equation*}%
with $\mathbb{H}_{-1}(x;q)=0$ and $\mathbb{H}_0(x;q) = 1$.
\end{theorem}



\begin{proof}
Shifting $n\rightarrow n+1$ in (\ref{eq:risingHn-1}) and multiplying by $%
\Theta_{1,2,n}(x)$ we get

\begin{align}
\Theta_{1,2,n}(x) \left(\hat{a}^\dagger_{n+1} \mathbb{H}_n(x;q)\right) = & 
\hspace{0.1cm}\Theta_{1,2,n}(x)\left[\Theta_{1,2,n+1}(x)\mathscr{D}%
_q-\Theta_{1,4,n+1} (x)I_d\right]\mathbb{H}_n(x;q) =  \notag \\
= & \hspace{0.1cm} \Theta_{1,2,n}(x)\Theta_{4,2,n+1}(x)\mathbb{H}_{n+1}(x;q)
\label{eq:3trr1}
\end{align}
on the other hand, multiplying (\ref{eq:lowerHn}) by $\Theta_{1,2,n+1}(x)$
we obtain 
\begin{align}
\Theta_{1,2,n+1}(x)\left(\hat{a}_n\mathbb{H}_n(x;q)\right) = & \hspace{0.1cm}
\Theta_{1,2,n+1}(x)\left[\Theta_{1,2,n}(x)\mathscr{D}_q-\Theta_{3,2,n}(x)I_d %
\right]\mathbb{H}_n(x;q) =  \notag \\
= & \hspace{0.1cm}\Theta_{1,2,n+1}(x)\Theta_{1,3,n}(x)\mathbb{H}_{n-1}(x;q)
\label{eq:3trr2}
\end{align}
then, subtracting (\ref{eq:3trr2})-(\ref{eq:3trr1}) and after some
elementary manipulations we are able to get 
\begin{align}
\Theta_{1,2,n}(x)\Theta_{4,2,n+1}(x)\mathbb{H}_{n+1}(x;q) = \hspace{0.1cm}
[\Theta_{1,2,n+1}(x)\Theta_{3,2,n}&(x)-\Theta_{1,2,n}(x)\Theta_{1,4,n+1}(x)]%
\mathbb{H}_n(x;q)  \notag \\
& + \Theta_{1,2,n+1}(x)\Theta_{1,3,n}(x)\mathbb{H}_{n-1}(x;q)
\end{align}
which is (\ref{eq:3TRR}) for the set of coefficients $\alpha_n(x)$, $%
\beta_n(x)$ and $\gamma_n(x)$ given in (\ref{eq:3TRRcoeffs})
\end{proof}



\begin{remark}
It is not difficult to show that the inner product (\ref{piSob}) is
symmetric with respect to the product by the polynomial $(x\boxminus_q
\alpha)^{j+1}$. Following similar steps to those found in \cite[Proposition 4%
]{MR1990}, the previous statement leads us to a $2j+3$ recursion formula
satisfied by the sequence $\left\{\mathbb{H}_n(x;q)\right\}_{n\ge 0}$ of the
form

\begin{equation}
(x\boxminus_q \alpha)^{j+1}\mathbb{H}_n(x;q) = \displaystyle\sum_{\nu =
n-j-1}^{n+j+1}d_{n,\nu}\mathbb{H}_\nu(x;q)  \label{eq:recurrence_dnk}
\end{equation}
where the coefficients $d_{n,k}$ satisfy

\begin{equation*}
d_{n,k} = \frac{\langle (x\boxminus_q \alpha)^{j+1}\mathbb{H}_n(x;q), 
\mathbb{H}_k(x;q)\rangle}{\langle\mathbb{H}_k,\mathbb{H}_{k}\rangle} = \frac{%
\langle \mathbb{H}_n(x;q),(x\boxminus_q \alpha)^{j+1} \mathbb{H}%
_k(x;q)\rangle}{\langle\mathbb{H}_k,\mathbb{H}_{k}\rangle}
\end{equation*}
with $d_{n,n+j+1} = 1$ and 
\begin{equation*}
d_{n,n-j-1} = \frac{\langle \mathbb{H}_n(x;q),(x\boxminus_q \alpha)^{j+1} 
\mathbb{H}_{n-j-1}(x;q)\rangle}{\langle\mathbb{H}_{n-j-1},\mathbb{H}%
_{n-j-1}\rangle} = \frac{\langle\mathbb{H}_n(x;q),x^n\rangle}{\langle\mathbb{%
H}_{n-j-1},\mathbb{H}_{n-j-1}\rangle}=\frac{\langle\mathbb{H}_{n},\mathbb{H}%
_{n}\rangle}{\langle\mathbb{H}_{n-j-1},\mathbb{H}_{n-j-1}\rangle}>0
\end{equation*}
these coefficients contain the same amount of information that the ones
given in (\ref{eq:3TRRcoeffs}). Then, the three-term recurrence formula (\ref%
{eq:3TRR}) allows us to obtain $\mathbb{H}_{n+1}(x;q)$ knowing two
consecutive polynomials $\mathbb{H}_{n}(x;q)$ and $\mathbb{H}_{n-1}(x;q)$
through a recursion involving fewer terms than (\ref{eq:recurrence_dnk}).
Nevertheless, we have to deal with complicated non-constant coefficients.
\end{remark}

\section{Distribution and asymptotic behaviour of the zeros}

\label{sec:zeros} In this section we derive several results concerning the
zeros of the Sobolev-type $q-$Hermite I orthogonal polynomials $\mathbb{H}%
_n(x;q)$. We first obtain a sufficient condition that ensures these zeros
are real. Additionally, we study their location with respect to the interval 
$(-1,1)$ and with respect to the mass point $\alpha$. Next, we investigate
interlacing relations with the zeros of the $q-$Hermite I polynomials.
Moreover we analyze the behaviour of the zeros of $\mathbb{H}_n(x;q)$ as a
function of $\lambda$ when $\lambda$ tends from zero to infinity. Finally we
support the results obtained in this section with some interesting numerical
experiments.\newline

Concerning the classical $q$-Hermite I polynomials $H_n(x;q)$, we know that
their zeros are all real, simple and lie in (-1,1). For the considered
Sobolev polynomials, we have that $\mathbb{H}_n(x;q)$ can have complex zeros
as shown in Fig.\ref{fig:complexzeros}. Here it is shown the real and
complex zeros of $\mathbb{H}_4(x;q)$ for $\alpha = 0.25$, $q=0.6$, $j = 2$
and $\lambda=100$ whose numerical values are $-0.144763, 0.858505, -0.778637
\pm 0.221654 i$.



\begin{figure}[h]
\centering
\includegraphics[height=0.30\textwidth]{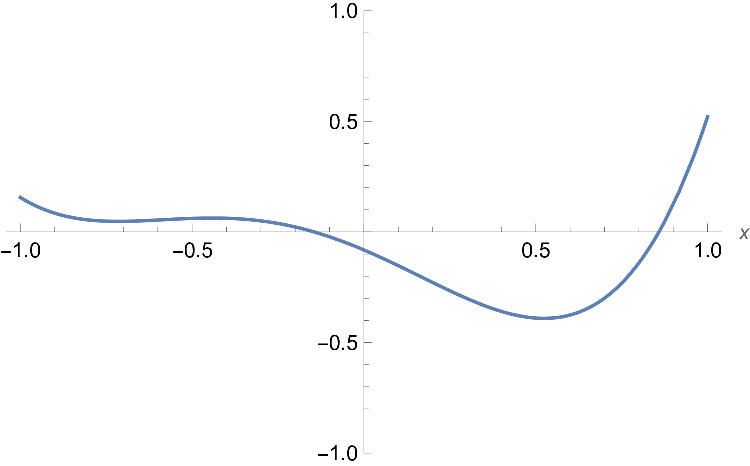} %
\includegraphics[width=0.45\textwidth]{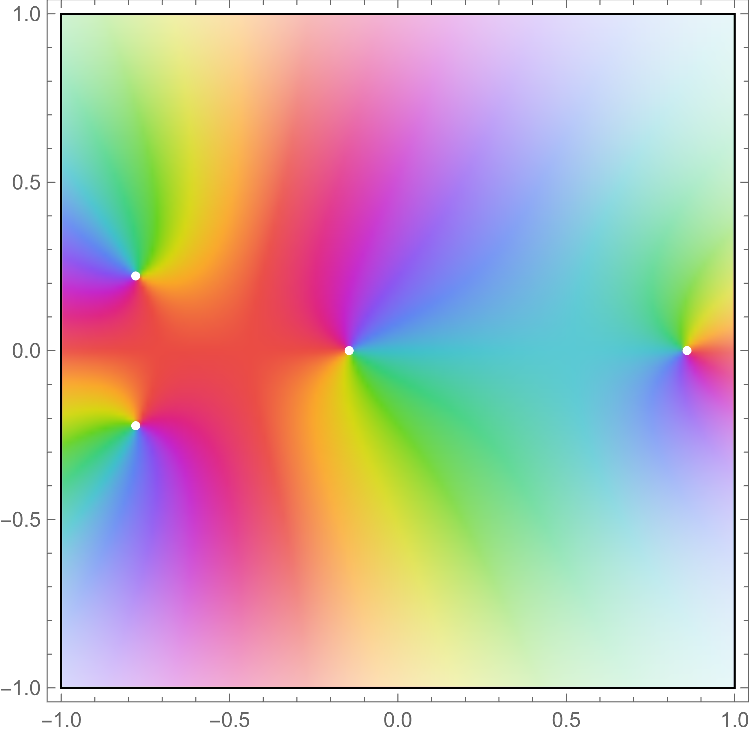}
\caption{Graph of $\mathbb{H}_4(x;q)$ for $\protect\alpha = 0.25$, $q=0.6$, $%
j = 2$ and $\protect\lambda = 100$ (left). Location of the real and complex
zeros of this polynomial (right).}
\label{fig:complexzeros}
\end{figure}



The fact of having complex zeros largely depends on the choice of $\alpha$.
Now we present a result ensuring that, for a specific configuration of the
parameters of the system, the $n$ zeros of $\mathbb{H}_n(x;q)$ are all real
and we locate them with respect to the interval $(-1,1)$.\newline



\begin{theorem}
Let $\lambda >0$, $j\geq 1$, $q\in(0,1)$ and $\alpha \notin (-1,1)$ such
that $q^j\alpha\notin (-1,1)$. Then for $n\geq j+1$ the polynomial $\mathbb{H%
}_n(x;q)$ has $n$ real simple zeros, at most one of them outside $(-1,1)$.
If $q^j\alpha <-1$ and $\mathbb{H}_n(x;q)$ has a zero outside $(-1,1)$, then
the zero can not be in $[1,\infty)$. If $q^j\alpha >1$ and $\mathbb{H}%
_n(x;q) $ has a zero outside $(-1,1)$, then the zero can not lie in $%
(-\infty,-1]$.
\end{theorem}



\begin{proof}
Let $\eta_1, \eta_2,...,\eta_k$ denote the zeros of $\mathbb{H}_n(x;q)$ of
odd multiplicity in $(-1,1)$. Put $\rho(x)=(x-\eta_1)(x-\eta_2)\cdots
(x-\eta_k)$. Consider $q^j\alpha< -1$. Then $\rho(x)\mathbb{H}_{n}(x;q)$ and 
$(x-q^j\alpha)\rho(x)\mathbb{H}_{n}(x;q)$ do not change sign on $[-1,1]$.
Suppose $deg(\rho)\leq n-2$. Hence $\langle \rho,\mathbb{H}_n\rangle_\lambda
=0$ and $\langle (x-q^j\alpha)\rho,\mathbb{H}_n\rangle_\lambda = 0$, which
respectively implies

\begin{equation}
\displaystyle\int_{-1}^1\rho(x)\mathbb{H}_n(x;q) d_q\mu(x) +\lambda \left(%
\mathscr{D}^{j}_q\rho\right)(\alpha)\left(\mathscr{D}^{j}_q\mathbb{H}%
_n\right)(\alpha;q) = 0  \label{eq:innerP=0}
\end{equation}
and

\begin{equation}
\displaystyle\int_{-1}^1(x-q^j\alpha)\rho(x)\mathbb{H}_n(x;q) d_q\mu(x)
+\lambda [j]_q\left(\mathscr{D}^{j-1}_q\rho\right)(\alpha)\left(\mathscr{D}%
^{j}_q\mathbb{H}_n\right)(\alpha;q) = 0
\end{equation}
then $\left(\mathscr{D}^{j}_q\rho\right)(\alpha)$ and $\left(\mathscr{D}%
^{j-1}_q\rho\right)(\alpha)$ should have the same sign. Since $\rho(x)$ is a
polynomial with positive leading coefficient, its $k$ zeros real, simple and
lying in $(-1,1)$, then $sgn\hspace{0.1cm} \rho(\alpha) = (-1)^k$ and $sgn%
\hspace{0.1cm} (\mathscr{D}^j_q \rho)(\alpha) = (-1)^j\hspace{0.1cm} sgn%
\hspace{0.1cm}\rho(\alpha)$ for all $q^j\alpha < -1$. Therefore $\left(%
\mathscr{D}^{j}_q\rho\right)(\alpha)$ and $\left(\mathscr{D}%
^{j-1}_q\rho\right)(\alpha)$ can not have the same sign. We conclude that $%
deg(\rho)=n$ or $deg(\rho)=n-1$. Assume now that $deg(\rho)=n-1$. Then there
is a zero outside $(-1,1)$. From (\ref{eq:innerP=0}) we conclude that $%
\rho(\alpha)\mathbb{H}_n(\alpha;q)$ and $\left(\mathscr{D}%
^{j}_q\rho\right)(\alpha)\left(\mathscr{D}^{j}_q\mathbb{H}%
_n\right)(\alpha;q) $ should differ in sign. For $q^j\alpha< -1$, $sgn%
\hspace{0.1cm} (\mathscr{D}^j_q \rho)(\alpha)=(-1)^j\hspace{0.1cm} sgn%
\hspace{0.1cm}\rho(\alpha)$. Hence $sgn\hspace{0.1cm}(\mathscr{D}^j_q\mathbb{%
H}_n)(\alpha;q) = (-1)^{j+1}\hspace{0.1cm}sgn\hspace{0.1cm}\mathbb{H}%
_n(\alpha;q)$. This implies that the zero outside $(-1,1)$ can not lie in $%
[1,\infty).$\newline



Consider now $\alpha >1$, $q\in(0,1)$ and $j\geq 1$ such that $q^j\alpha>1$.
Then $\rho(x)\mathbb{H}_n(x;q)$ and $(q^j\alpha-x)\rho(x)\mathbb{H}_n(x;q)$
do not change sign in $[-1,1]$. Assume $deg(\rho)\leq n-2$. Hence $\langle
\rho,\mathbb{H}_n\rangle_\lambda =0$ and $\langle (q^j\alpha-x)\rho,\mathbb{H%
}_n\rangle_\lambda = 0$. We conclude now that $\left(\mathscr{D}%
^{j}_q\rho\right)(\alpha)$ and $\left(\mathscr{D}^{j-1}_q\rho\right)(\alpha)$
should differ in sign. This statement does not hold since $sgn \hspace{0.1cm}%
\rho(\alpha) = sgn\hspace{0.1cm}(\mathscr{D}^j_q)(\alpha)=1$ for all $%
q^j\alpha >1$. Hence $deg(\rho)=n$ or $deg(\rho)=n-1$. Assume $deg(\rho)=n-1$%
. Then $\rho(\alpha)\mathbb{H}_n(\alpha;q)$ and $\left(\mathscr{D}%
^{j}_q\rho\right)(\alpha)\left(\mathscr{D}^{j}_q\mathbb{H}%
_n\right)(\alpha;q) $ should differ in sign. Since $sgn \hspace{0.1cm}%
\rho(\alpha) = sgn\hspace{0.1cm}(\mathscr{D}^j_q)(\alpha)=1$, we arrive to $%
sgn\hspace{0.1cm}(\mathscr{D}^j_q\mathbb{H}_n)(\alpha;q) = -sgn \hspace{0.1cm%
}\mathbb{H}_n(\alpha;q)$. This implies that the zero outside $(-1,1)$ can
not lie in $(-\infty,-1]$. Thus for $q^j\alpha\notin (-1,1)$ we have $%
deg(\rho) = n$ or $deg(\rho)=n-1$.


\end{proof}



\begin{remark}
Since $\mathbb{H}_n(x;q)=H_n(x;q)$ for $0\leq n\leq j$, $\mathbb{H}_n(x;q)$
can have a zero outside $(-1,1)$ for $n\geq j+1$.
\end{remark}

Next we provide the explicit value $\lambda_0$ of the mass $\lambda$ such
that for $\lambda>\lambda_0$ at most one zero of $\mathbb{H}_n(x;q)$ is
located outside $(-1,1)$



\begin{corollary}
\label{col:ceros} Let $\lambda>0$, $q\in(0,1)$, $\alpha\notin(-1,1)$ and $%
j\leq n-1 $ such that $q^j\alpha\notin(-1,1)$. Then, the following
statements hold 

\begin{itemize}
\item if $q^j\alpha<-1$, then the smallest zero of $\mathbb{H}_n(x;q)$, $%
\eta_{n,1}=\eta_{n,1}(\lambda)$, satisfies 
\begin{align}
&\eta_{n,1}>-1, \quad \text{for $\lambda<\lambda_0$}  \notag \\
&\eta_{n,1}=-1, \quad \text{for $\lambda=\lambda_0$}  \label{eq:ceros<} \\
&\eta_{n,1}<-1, \quad \text{for $\lambda>\lambda_0$}  \notag
\end{align}
where 
\begin{equation*}
\lambda_0=\lambda_0(n,\alpha,j,q) = \left(\frac{[n]_q^{(j)}H_{n-j}(\alpha;q)%
}{H_n(-1;q)}K_{n-1,q}^{(0,j)}(-1,\alpha)-K_{n-1,q}^{(j,j)}(\alpha,\alpha)%
\right)^{-1}>0
\end{equation*}

\item if $q^j\alpha>1$, then the largest zero, $\eta_{n,n} =
\eta_{n,n}(\lambda)$, satisfies 
\begin{align}
&\eta_{n,n}<1, \quad \text{for $\lambda<\lambda_0$}  \notag \\
&\eta_{n,n}=1, \quad \text{for $\lambda=\lambda_0$}  \label{eq:ceros>} \\
&\eta_{n,n}>1, \quad \text{for $\lambda>\lambda_0$}  \notag
\end{align}
with 
\begin{equation*}
\lambda_{0} = \lambda_0(n,\alpha,j,q)= \left(\frac{[n]_q^{(j)}H_{n-j}(%
\alpha;q)}{H_n(1;q)}K_{n-1,q}^{(0,j)}(1,\alpha)-K_{n-1,q}^{(j,j)}(\alpha,%
\alpha)\right)^{-1}>0
\end{equation*}
\end{itemize}
\end{corollary}



\begin{proof}
It is suffices to use (\ref{ConxF1}) together with the fact that $\mathbb{H}%
_n(\tau;q) = 0$, with $\tau = -1, 1$, if and only if $\lambda = \lambda_0$ 
\begin{equation*}
\mathbb{H}_{n}(\tau;q)=H_{n}(\tau;q)-\lambda_0 \frac{\lbrack
n]_{q}^{(j)}H_{n-j}(\alpha ;q)}{1+\lambda_0 K_{n-1,q}^{(j,j)}(\alpha ,\alpha
)}K_{n-1,q}^{(0,j)}(\tau,\alpha )=0.
\end{equation*}
therefore 
\begin{equation*}
\lambda_0 = \lambda_0(n,\alpha,j,q) = \left(\frac{[n]_q^{(j)}H_{n-j}(%
\alpha;q)}{H_n(\tau;q)}K_{n-1,q}^{(0,j)}(\tau,\alpha)-K_{n-1,q}^{(j,j)}(%
\alpha,\alpha)\right)^{-1}
\end{equation*}
\end{proof}



notice that for $j=0$, the inner product (\ref{piSob}) is standard. As a
consequence, the zeros of $\mathbb{H}_n(x;q)$ are located inside the convex
hull of the support of the measure. In that case, the zero located outside $%
(-1,1)$ is attracted by the mass point $\alpha$. For $j=0$, $\alpha$
captures either the smallest or largest zero as $\lambda \rightarrow \infty$%
. For $j\geq 1$ the situation is completely different, being possible to
find the value of the zero outside $(-1,1)$ beyond the value of $\alpha$.
The value of $\lambda$, namely $\lambda_\alpha$, for which we locate the
smallest or largest zero of $\mathbb{H}_n(x;q)$ at the mass point $\alpha$
can easily be found



\begin{corollary}
Let $\lambda>0$, $q\in(0,1)$, $\alpha\notin (-1,1)$ and $j\leq n-1$ such
that $q^j\alpha\notin(-1,1)$. Then, the following statements hold

\begin{itemize}
\item if $q^j\alpha<-1$, then the smallest zero $\eta_{n,1}$ satisfies 
\begin{align}
&\eta_{n,1}>\alpha, \quad \text{for $\lambda<\lambda_\alpha$}  \notag \\
&\eta_{n,1}=\alpha, \quad \text{for $\lambda=\lambda_\alpha$}
\label{eq:ceros<a} \\
&\eta_{n,1}<\alpha, \quad \text{for $\lambda>\lambda_\alpha$}  \notag
\end{align}

\item if $q^j\alpha>1$, then the largest zero $\eta_{n,n}$ satisfies 
\begin{align}
&\eta_{n,n}<\alpha, \quad \text{for $\lambda<\lambda_\alpha$}  \notag \\
&\eta_{n,n}=\alpha, \quad \text{for $\lambda=\lambda_\alpha$}
\label{eq:ceros>a} \\
&\eta_{n,n}>\alpha, \quad \text{for $\lambda>\lambda_\alpha$}  \notag
\end{align}
where 
\begin{equation*}
\lambda_\alpha=\lambda_\alpha(n,\alpha,j,q) = \left(\frac{%
[n]_q^{(j)}H_{n-j}(\alpha;q)}{H_n(\alpha;q)}K_{n-1,q}^{(0,j)}(\alpha,%
\alpha)-K_{n-1,q}^{(j,j)}(\alpha,\alpha)\right)^{-1}>0
\end{equation*}
\end{itemize}

\label{col:lambdalfa}
\end{corollary}



\begin{proof}
The proof is exactly the same as in corollary \ref{col:ceros} taking into
account that $\mathbb{H}_n(\alpha,q) = 0$ if and only if $%
\lambda=\lambda_\alpha$.
\end{proof}



\begin{proposition}
Let $\left\{\mathbb{H}_n(x;q)\right\}_{n>0}$ be the sequence of $q-$ Hermite
I Sobolev-type orthogonal polynomials. Then, the following statement holds 
\begin{equation}
(\mathscr{D}^j_q\mathbb{H}_n)(\alpha) = \left(1-\frac{\lambda
K^{(j,j)}_{n-1}(\alpha,\alpha)}{1+\lambda K^{(j,j)}_{n-1}(\alpha,\alpha)}%
\right)(\mathscr{D}^j_q H_n)(\alpha)  \label{eq:propDq}
\end{equation}
\end{proposition}



\begin{proof}
Applying $\mathscr{D}^j_q$ to (\ref{eq:DqconexionI}) and evaluating at $%
\alpha$ we arrive to the desired result.
\end{proof}



Now we present a theorem stating the interlacing of the zeros of $\mathbb{H}%
_n(x;q)$ and $H_n(x;q)$



\begin{theorem}
Let $\lambda>0$, $q\in (0,1)$, $j\le n-1$ and $\alpha\notin(-1,1)$ such that 
$q^j\alpha \notin (-1,1)$. Let $\eta_1,\eta_2,...,\eta_n$ and $%
x_1,x_2,...,x_n$ denote the zeros of $\mathbb{H}_n(x)$ and $H_n(x)$,
respectively. Then, the following interlacing relations are satisfied

\begin{itemize}
\item If $q^j\alpha<-1$ then $\eta_1<x_1$ and 
\begin{equation*}
x_i<\eta_{i+1}<x_{i+1}, \quad i = 1,2,...n-1
\end{equation*}

\item If $q^j\alpha>1$ then $\eta_n>x_n$ and 
\begin{equation*}
x_i< \eta_i<x_{i+1}, \quad i = 1,2,...,n-1
\end{equation*}
\end{itemize}
\end{theorem}



\begin{proof}
The proof here is based on the the Gaussian quadrature rule 
\begin{equation*}
\displaystyle\sum_{i=1}^n \omega_i \Pi_{2n-1}(x_i) = \displaystyle%
\int_{-1}^{1}\Pi_{2n-1}(x) d_q\mu(x)
\end{equation*}
where $\omega_i>0$ for all $i=1,2,...,n$ and $\Pi_{2n-1}(x)$ is a polynomial
of degree $2n-1$. Take $\Pi_{2n-1}(x) = \mathbb{H}_n(x)H_n(qx)/(qx-x_i)$.
Then, with (\ref{eq:propDq})

\begin{align*}
\omega_i\mathbb{H}_n(x_i)(\mathscr{D}_qH_n)(x_i) =& \displaystyle%
\int_{-1}^{1}\Pi_{2n-1}(x) d_q\mu(x) \\
=& \langle \mathbb{H}_n,\frac{H_n(qx)}{(qx-x_i)}\rangle_\lambda - \lambda(%
\mathscr{D}^j_q\mathbb{H}_n)(\alpha)\mathscr{D}^j_q \left(\frac{H_n(qx)}{%
qx-x_i}\right)(\alpha) \\
=& - \lambda(\mathscr{D}^j_q\mathbb{H}_n)(\alpha)\mathscr{D}^j_q \left(\frac{%
H_n(qx)}{qx-x_i}\right)(\alpha) \\
=&-\lambda \left(1-\frac{\lambda K^{(j,j)}_{n-1}(\alpha,\alpha)}{1+\lambda
K^{(j,j)}_{n-1}(\alpha,\alpha)}\right)(\mathscr{D}^j_qH_n)(\alpha)\mathscr{D}%
^j_q \left(\frac{H_n(qx)}{qx-x_i}\right)(\alpha)
\end{align*}
Since $H_n(qx)/(qx-x_i)$ is a monic polynomial of degree $n-1$ with $n-1$
simple real zeros in $(-1,1)$ we have 
\begin{equation*}
sgn\hspace{0.1cm} \mathscr{D}^j_q \left(\frac{H_n(qx)}{qx-x_i}%
\right)(\alpha) = \left\{ 
\begin{array}{lll}
(-1)^{n-1-j} & if & q^j\alpha<-1 \\ 
&  &  \\ 
1 & if & q^j\alpha>1%
\end{array}
\right.
\end{equation*}
moreover, 
\begin{equation*}
sgn\hspace{0.1cm} (\mathscr{D}^j_qH_{n})(\alpha) = \left\{ 
\begin{array}{lll}
(-1)^{n-j} & if & q^j\alpha <-1 \\ 
&  &  \\ 
1 & if & q^j\alpha >1%
\end{array}
\right.
\end{equation*}

and, since $H_n(x;q)$ is a classical sequence, 
\begin{equation*}
sgn\hspace{0.1cm} (\mathscr{D}_qH_{n})(x_i) = (-1)^{n-i}
\end{equation*}

then, for $q^j\alpha \notin (-1,1)$ we finally obtain 
\begin{equation*}
sgn \hspace{0.1cm} \mathbb{H}_n(x_i;q) = (-1)^{n-i}
\end{equation*}
which means that $\mathbb{H}_n(x;q)$ changes sign in each interval $%
(x_i,x_{i+1})$ for $i=1,2,...n-1$.\newline
\end{proof}



Our goal now is to obtain some results regarding the behaviour of the zeros
of $\mathbb{H}_n(x;q)$. For this purpose we base our analysis on the
following lemma, which addresses the behavior and asymptotic characteristics
of the zeros of a linear combination of two polynomials of degree $n$
exhibiting interlacing properties of their zeros. In the following we will
refer to this lemma as the \textit{\textbf{Interlacing lemma}}, and although
the proof is omitted here (see \cite[Lemma 1]{Bracc}), it plays a crucial
role in our analysis. 



\begin{lemma}
(\textbf{Interlacing lemma}) Let $p_n(x) = a(x-x_1)\cdots(x-x_n)$ and $%
q_n(x) = b(x-y_1)\cdots(x-y_n)$ be polynomials with real and simple zeros,
where $a$ and $b$ are real positive constants. If 
\begin{equation*}
y_1<x_1<\cdots< y_n<x_n
\end{equation*}
then, for any real constant $c>0$, the polynomial 
\begin{equation}
f_n(x) = p_n(x) +cq_n(x)  \label{eq:interlacinglemma}
\end{equation}
has $n$ real zeros $\eta_1,\cdots,\eta_n$ which interlace with the zeros of $%
p_n(x)$ and $q_n(x)$ as follows 
\begin{equation*}
y_1<\eta_1<x_1 \cdots y_n<\eta_n<x_n
\end{equation*}
Moreover, each $\eta_k = \eta_k(c)$ is a decreasing function of $c$, and for
each $k = 1,...,n$, 
\begin{equation*}
\begin{array}{ccc}
\displaystyle\lim_{c\rightarrow \infty} \eta_k = y_k & and & \displaystyle%
\lim_{c\rightarrow \infty}c[\eta_k-y_k] = -\frac{p_n(y_k)}{q^{\prime }_n(y_k)%
}%
\end{array}%
\end{equation*}
\end{lemma}



In order to apply the interlacing lemma and to study the dynamics of the
zeros of the q-Hermite I-Sobolev type polynomials when $\lambda$ tends to
infinity, we introduce the limiting polynomial

\begin{equation}
\mathbb{G}_n(x;q) = \displaystyle\lim_{\lambda \rightarrow \infty} \mathbb{H}%
_n(x;q) = H_n(x) - \frac{[n]^{(j)}_q H_{n-j}(\alpha;q)}{K^{(j,j)}_{n-1,q}(%
\alpha,\alpha)}K^{(0,j)}_{n-1,q}(x;\alpha)  \label{eq:limitingG}
\end{equation}

next, it is useful to normalize the connection formula (\ref{ConxF1}) as
follows



\begin{proposition}
The sequence of polynomials $\mathcal{H}_n(x;q)= \left(1+\lambda
K_{n-1,q}^{(j,j)}(\alpha,\alpha)\right)\mathbb{H}_n(x;q)$ satisfies the
identity

\begin{equation}
\mathcal{H}_n(x;q) = H_n(x;q) +\lambda K_{n-1,q}^{(j,j)}(\alpha,\alpha) 
\mathbb{G}_n(x;q)  \label{eq:Hinterlacing}
\end{equation}
\end{proposition}



\begin{proof}
From (\ref{eq:limitingG}), we have

\begin{equation*}
K_{n-1,q}^{(0,j)}(x,\alpha) = \frac{K_{n-1,q}^{(j,j)}(\alpha,\alpha)}{%
[n]_q^{(j)}H_{n-j}(x;q)}\left[H_n(x;q)-\mathbb{G}_n(x;q)\right]
\end{equation*}
combining the above expression together with (\ref{ConxF1}) we get 
\begin{equation*}
\mathbb{H}_n(x;q) = H_n(x;q) - \frac{\lambda K_n^{(j,j)}(\alpha,\alpha)}{%
1+\lambda K_n^{(j,j)}(\alpha,\alpha)}\left[H_n(x;q)-\mathbb{G}_n(x;q)\right]
\end{equation*}
then, multiplying the above equation by $1+\lambda
K_n^{(j,j)}(\alpha,\alpha) $ one is able to get 
\begin{equation}
\left(1+\lambda K_n^{(j,j)}(\alpha,\alpha)\right)\mathbb{H}_n(x;q) =
H_n(x;q) +\lambda K_{n-1,q}^{(j,j)}(\alpha,\alpha) \mathbb{G}_n(x;q)
\end{equation}
which yields (\ref{eq:Hinterlacing}). Finally, from (\ref{eq:KernelDij}) 
\begin{equation*}
K_{n-1,q}^{(j,j)}(\alpha,\alpha) = \displaystyle\sum_{k=0}^{n-1}\frac{%
\left([k]_q^{(j)}H_{k-j}(\alpha;q)\right)^2}{||H_k||^2} >0
\end{equation*}
for $n\ge j+1$.
\end{proof}



In view of the positiveness of $K_{n-1,q}^{(j,j)}(\alpha,\alpha)$ for $n\ge
j+1$ and the interlacing of the zeros of $H_n(x;q)$ and $\mathbb{G}_n(x;q)$,
which is a Sobolev polynomial, from (\ref{eq:Hinterlacing}) we are in the
hypothesis of the Interlacing lemma and we immediately conclude the
following results.



\begin{theorem}
Let $\lambda>0$, $0<q<1$, $\alpha\notin(-1,1)$ and $j\in \mathbb{N}$ such
that $q^j\alpha\notin(-1,1)$ and $Q^{(n)}_j(x,\alpha)$ does not have its
zeros in $S_\mu$ for all $j\le n-1$ . Then, the zeros of $\mathbb{H}_n(x;q)$%
, $\left\{\eta_{n,r}\right\}_{r=1}^n$, the zeros of $H_n(x;q)$, $%
\left\{x_{n,r}\right\}_{r=1}^n$, and the zeros of $\mathbb{G}_n(x;q)$, $%
\left\{y_{n,r}\right\}_{r=1}^n$, satisfy the following interlacing relations

\begin{itemize}
\item if $q^j\alpha<-1$ then 
\begin{equation*}
y_{n,1}<\eta_{n,1}<x_{n,1}\cdots y_{n,n}<\eta_{n,n}<x_{n,n}
\end{equation*}
for every $n\in\mathbb{N}$. Moreover, each $\eta_{n,l} = \eta_{n,l}(\lambda)$
is a decreasing function of $\lambda$ and for each $l=1,...,n$ 
\begin{equation*}
\displaystyle\lim_{\lambda\rightarrow \infty}\eta_{n,l}(\lambda)=y_{n,l} \\
\end{equation*}
and 
\begin{equation*}
\displaystyle\lim_{\lambda\rightarrow\infty}\lambda[\eta_{n,l}(%
\lambda)-y_{n,l}]=-\frac{H_n(y_{n,l};q)}{\mathbb{G}^{\prime }_n(y_{n,l};q)}
\end{equation*}


\item if $q^j\alpha>1$ then 
\begin{equation*}
x_{n,1}<\eta_{n,1}<y_{n,1}\cdots x_{n,n}<\eta_{n,n}<y_{n,n}
\end{equation*}
hold for every $n\in\mathbb{N}$. Moreover, each $\eta_{n,l} =
\eta_{n,l}(\lambda)$ is an increasing function of $\lambda$ and for each $%
l=1,...,n$ 
\begin{equation*}
\displaystyle\lim_{\lambda\rightarrow \infty}\eta_{n,l}(\lambda)=y_{n,l}
\end{equation*}
and 
\begin{equation*}
\displaystyle\lim_{\lambda\rightarrow\infty}\lambda[\eta_{n,l}(%
\lambda)-y_{n,l}]=-\frac{H_n(y_{n,l};q)}{\mathbb{G}^{\prime }_n(y_{n,l};q)}
\end{equation*}
\end{itemize}

\label{theorem:interlacing}
\end{theorem}



In order to support corollaries \ref{col:ceros} and \ref{col:lambdalfa}, we
show some numerical experiments using Wolfram Mathematica software$%
^\circledR $, dealing with the smallest and the largest zero of $\mathbb{H}%
_n(x;q)$, providing the exact values of $\lambda_\alpha$ for which the
smallest or largest zero is located beyond $\alpha$. When the smallest or
the largest zero of $\mathbb{H}_n(x;q)$ is smaller/greater than the value of 
$\alpha$, it is highlighted in bold type. It is clear that, if $\lambda = 0$%
, we recover the zeros of $H_n(x;q)$. Table \ref{table:ceros<} shows, for
some values of $\lambda$, the values of the two smallest zeros of $\mathbb{H}%
_n(x;q)$ for $n=7$, $j = 3$, $q = 2/3$ and $\alpha = -10$. For this
configuration we have $\lambda_0 = 6.32464 \cdot 10^{-15}$ and $%
\lambda_\alpha = 5.24853\cdot10^{-12}$.

Table \ref{table:ceros<} agrees with the analysis made in corollaries \ref%
{col:ceros} and \ref{col:lambdalfa}. For $\lambda =5\cdot 10^{-16}
<\lambda_0 $, the largest zero of $\mathbb{H}_7(x;q)$ for the configuration
given is inside $(-1,1)$. For $\lambda = 5\cdot 10^{-13}<\lambda_\alpha$,
this zero is outside $(-1,1)$ since $\lambda>\lambda_0$, but it has not yet
reached the value of $\alpha$. For $\lambda = \lambda_\alpha$, $\eta_{7,7}$
is exactly at $\alpha=-10$ and for $\lambda>\lambda_{\alpha}$ the value of
this zero is greater than $\alpha$.



\begin{table}[h!]
\centering
\begin{tabular}{llllll}
\hline
$\eta_{7,l}(\lambda)$ & 
\begin{tabular}{@{}l}
$\lambda = 5\cdot10^{-16}$%
\end{tabular}
& $\lambda =5\cdot 10^{-13}$ & $\lambda = \lambda_\alpha$ & $\lambda =
5\cdot10^{-10}$ & $\lambda = 5$ \\ \hline
$l=1$ & -0.99982 & -4.21644 & -10.00000 & \textbf{-11.68557} & \textbf{%
-11.70652} \\ 
$l=2$ & -0.65454 & -0.99319 & -0.99776 & -0.99778 & -0.99779 \\ \hline
\end{tabular}%
\caption{Values of the two smallest zeros of $\mathbb{H}_{7}(x;q)$ for $j=3$%
, $q=2/3$ and $\protect\alpha=-10$. When $\protect\eta_{7,1}<\protect\alpha$
the results are highlighted in bold type.}
\label{table:ceros<}
\end{table}

We have also performed a similar analysis when $q^j\alpha>1$. Table \ref%
{table:ceros>} shows the two largest zeros of $\mathbb{H}_n(x;q)$ for $n=10$%
, $j=8$, $q=4/5$ and $\alpha=15$. For this configuration we have $\lambda_0
= 2.51880\cdot 10^{-19}$ and $\lambda_\alpha=4.76010\cdot10^{-17}$.

\begin{table}[h!]
\centering
\begin{tabular}{llllll}
\hline
$\eta_{10,l}(\lambda)$ & 
\begin{tabular}{@{}l}
$\lambda =5\cdot10^{-21}$%
\end{tabular}
& $\lambda =5\cdot10^{-20}$ & $\lambda = \lambda_\alpha$ & $\lambda =
5\cdot10^{-15}$ & $\lambda = 1$ \\ \hline
$l=9$ & 0.79517 & 0.79564 & 0.99929 & 0.99932 & 0.99932 \\ 
$l=10$ & 0.99989 & 0.99991 & 15.00000 & \textbf{34.7206} & \textbf{37.1645}
\\ \hline
\end{tabular}%
\caption{Values of the first and the two largest zeros of $\mathbb{H}%
_{10}(x;q)$ for $\protect\alpha = 15$, $j=8$ and $q=4/5$. When $\protect\eta%
_{10,10}>\protect\alpha$ the results are highlighted in bold type.}
\label{table:ceros>}
\end{table}

Tables \ref{table:ceros<} and \ref{table:ceros>} perfectly support the set
of conditions (\ref{eq:ceros<a}) and (\ref{eq:ceros>a}) of corollary \ref%
{col:lambdalfa}, respectively. Also both tables show the monotonicity of $%
\eta_{n,k}$ as a function of $\lambda$ as stated in theorem \ref%
{theorem:interlacing}. Concerning the interlacing relations given in theorem %
\ref{theorem:interlacing}, table \ref{table:interlacing<} shows the values
of the zeros of $\mathbb{G}_n(x;q)$, $\mathbb{H}_n(x;q)$ and $H_n(x;q)$ for $%
n=8$, $j=5$, $\alpha=-120$, $\lambda = 10^{-20}$ and $q=2/5$.

\begin{table}[h!]
\centering
\begin{tabular}{|l|l|l|l|l|l|l|l|l|l|}
\hline
$k$ & $1$ & $2$ & $3$ & $4$ & $5$ & $6$ & $7$ & $8$ &  \\ \hline
$y_{8,k}$ & -128.11815 & -0.99999 & -0.39997 & -0.15387 & 0.00004 & 0.15389
& 0.39997 & 0.99999 &  \\ \hline
$\eta_{8,k}$ & -9.35956 & -0.99999 & -0.39997 & -0.15379 & 0.00028 & 0.15397
& 0.39997 & 0.99999 &  \\ \hline
$x_{8,k}$ & -0.99999 & -0.39999 & -0.15957 & -0.04781 & 0.04781 & 0.15957 & 
0.39999 & 0.99999 &  \\ \hline
\end{tabular}%
\caption{Zero interlacing of $\mathbb{G}_n(x;q)$, $\mathbb{H}_n(x;q)$ and $%
H_n(x;q)$ for $n=8$, $j=5$, $\protect\alpha=-120$, $\protect\lambda =
10^{-20}$ and $q=2/5$}
\label{table:interlacing<}
\end{table}

Similarly, table \ref{table:interlacing>} contains the interlacing relation
of $x_{n,k}$, $\eta_{n,k}$ and $y_{n,k}$ with $k=1,2,...n$ for $n=5$, $j=3$, 
$\alpha=15$, $\lambda=10^{-9}$ and $q=1/2$ 
\begin{table}[h!]
\centering
\begin{tabular}{|l|l|l|l|l|l|l|}
\hline
$k$ & $1$ & $2$ & $3$ & $4$ & $5$ &  \\ \hline
$x_{5,k}$ & -0.99938 & -0.46063 & 0.00000 & 0.46063 & 0.99938 &  \\ \hline
$\eta_{5,k}$ & -0.99182 & -0.33847 & 0.32873 & 0.99045 & 16.8962 &  \\ \hline
$y_{5,k}$ & -0.99178 & -0.33809 & 0.32913 & 0.99051 & 19.3393 &  \\ \hline
\end{tabular}%
\caption{Zero interlacing of $H_n(x;q)$, $\mathbb{H}_n(x;q)$ and $\mathbb{G}%
_n(x;q)$ for }
\label{table:interlacing>}
\end{table}



\section{Conclusions and further remarks}



In this paper, we have considered the $q$-Hermite I Sobolev-type polynomials
of higher order, and described important properties related to them such as
several structure relations, their hypergeometric representation derived
from the $q$-Hermite I polynomials, together with a three-term recurrence
relation of their elements. A deep analysis on the location and behaviour of
the zeros of this Sobolev sequence has also been performed.

In this last section we present some examples and further remarks on these
families of polynomials.\newline

First, we observe that the $q$-Hermite I polynomials are recovered when $%
\lambda =0$ for the $q$-Hermite I-Sobolev polynomials. Indeed, this can be
checked at many stages described in the work. It is direct to check that for
every $n\in \mathbb{N}$ the polynomial $\mathbb{H}_{n}(x;q)$ tends to $%
H_{n}(x;q)$ for $\lambda \rightarrow 0$ in view of (\ref{ConxF1}). Also, the
asymptotic behavior of the elements involved in the connection formulas when 
$\lambda $ approaches to 0 yields the result. It is straightforward to check
the asymptotic behavior of $\mathcal{E}_{k,n}$, $\mathcal{F}_{k,n}$ for $%
k=1,2$, and therefore of $\psi _{n}(x)$ for $\lambda \rightarrow 0$. In
addition to that, this asymptotic behavior can also be observed from the
hypergeometric representation of such polynomials. The first elements of $%
H_{n}(x;q)$ are

\begin{equation*}
H_0=1, \qquad H_1(x;q) = x
\end{equation*}
\begin{equation*}
H_2(x;q) = x^2+q-1, \qquad H_3(x;q) = x^3+(q^3-1)x
\end{equation*}
\begin{equation*}
H_4(x;q) = x^4+\left(q^5+q^3-q^2-1\right) x^2+q^6-q^5-q^3+q^2
\end{equation*}
\begin{equation*}
H_5(x;q) = x^5+\left(q^7+q^5-q^2-1\right)
x^3+\left(q^{10}-q^7-q^5+q^2\right) x
\end{equation*}

Let us now fix the parameters defining one of the families of Sobolev-type
orthogonal polynomials with $\alpha=3$, $q=3/5$. We also choose $j=2$, which
entails that the $q$-derivatives of second order evaluated at $\alpha$ are
involved in the definition of the sequence of polynomials. As we have stated
above, for $\lambda=0$ we recover the $q-$Hermite I polynomials. For $%
\lambda>0$ and using the novel three term recursion formula with non
constant coefficients in (\ref{eq:3TRR}), we show the first elements of the
sequence $\mathbb{H}_{n}(x;q)$ for the previous configuration and initial
conditions given by $\mathbb{H}_0(x;q)=1$ and $\mathbb{H}_{-1}(x;q)=0$ 
\begin{equation*}
\mathbb{H}_0(x;q) = 1
\end{equation*}
\begin{equation*}
\mathbb{H}_{1}(x;q) = x, \qquad \mathbb{H}_2(x;q) = x^2-\frac{2}{5}
\end{equation*}
\begin{equation*}
\mathbb{H}_3(x;q) \approx H_3(x;q)-\frac{65.5362 \lambda \left(x^2-0.4\right)%
}{1+11.1456 \lambda }, \qquad \mathbb{H}_{4}\approx H_4(x;q)-\frac{5323
\lambda \left(x^3+0.048 x^2-0.784 x-0.0192\right)}{1+1376.48 \lambda }
\end{equation*}
\begin{equation*}
\mathbb{H}_5(x;q) \approx H_5(x;q)-\frac{372144 \lambda \left( x^4+0.0482233
x^3-1.06393 x^2-0.037807 x+0.11197\right)}{1+111759 \lambda }
\end{equation*}



Figure \ref{fig:qHS_plots} illustrates this sequence of orthogonal
polynomials in the particular case of $\lambda=1$.



\begin{figure}[h]
\centering
\includegraphics[width=0.8\textwidth]{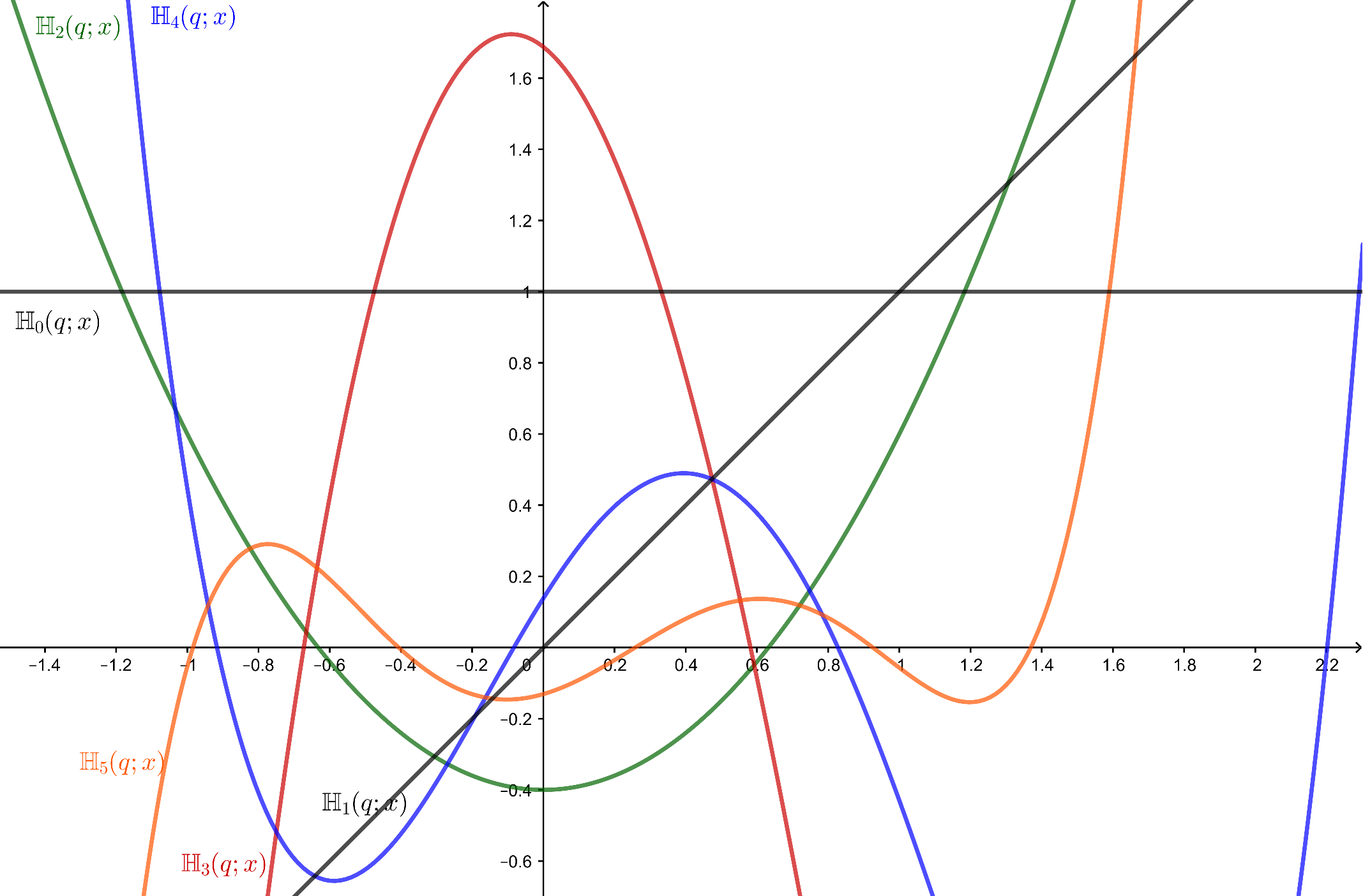}
\caption{The polynomials $\mathbb{H}_k(x;q)$ for $k=0,\ldots,5$ under the
previous settings}
\label{fig:qHS_plots}
\end{figure}



Observe that the parity of the elements of the family is not maintained in
the Sobolev settings.

In addition to this, one can check that $\mathbb{H}_k(x;q)$ coincides with $%
H_k(x;q)$ for $k=0,1,2$. This is a general property which is worth remarking
and which comes from the very definition of the polynomials. Indeed, observe
from (\ref{piSob}) that the Sobolev part of the inner product does not
appear when considering functions $f$ or $g$ being polynomials of degree at
most $j-1$ because of the following general property:



\begin{lemma}
\label{e11} Let $p(x)\in\mathbb{R}[x]$ of degree at most $j-1$. Then ${%
\mathscr D}^{j}_{q}p\equiv 0$.
\end{lemma}



\begin{proof}
It is sufficient to prove the result for the monomials $p_k(x)=x^k$ for $%
k=0,1,\ldots,j-1$. We observe that ${\mathscr D}^{j}_{q}1\equiv 0$ and 
\begin{equation*}
{\mathscr D}^{j}_{q}x^k=\frac{1-q^{k}}{1-q}x^{k-1},\qquad k\ge 1.
\end{equation*}
This allows to conclude the result.
\end{proof}



Applying Lemma~\ref{e11}, we observe from the definition of the $n$-th
reproducing kernel associated to the $q$-Hermite I polynomials that $%
K_{j-1,q}^{(0,j)}(x,\alpha)\equiv0$. Therefore, in view of (\ref{ConxF1}) we
also obtain that $\mathbb{H}_j(x;q)=H_j(x;q)$.

As a consequence, we have the following result.



\begin{corollary}
$\mathbb{H}_k(x;q)\equiv H_k(x;q)$ for every $0\le k\le j$.
\end{corollary}

Finally, the condition $q^j\alpha\notin(-1,1)$ is crucial for conducting the
analysis of the zeros of $\mathbb{H}_n(x;q)$ in this type of $q$-lattices,
ensuring that they are all real. Note that this condition must be satisfied
throughout the paper. For instance in (\ref{Kernel0j}), where, due to the
presence of operator $(x\boxminus_q\alpha)^{j+1}$ in the denominator of $%
\mathcal{A}(x,\alpha)$ and $\mathcal{B}(x,\alpha)$, the condition $q^j\alpha
\notin(-1,1)$ must be ensured. 





\section*{Acknowledgements}


The work of E. J. Huertas and A. Lastra has been supported by Dirección
General de Investigación e Innovación, Consejería de Educación e Investigació%
n of the Comunidad de Madrid (Spain) and Universidad de Alcalá, under grant
CM/JIN/2021-014, \textit{Proyectos de I+D para Jóvenes Investigadores de la
Universidad de Alcalá 2021}, and the Ministerio de Ciencia e Innovació%
n-Agencia Estatal de Investigación MCIN/AEI/10.13039/501100011033 and the
European Union ``NextGenerationEU''/PRTR, under grant TED2021-129813A-I00.%
\newline

This research was conducted while E. J. Huertas was visiting the ICMAT
(Instituto de Ciencias Matem\'aticas), from jan-2023 to jan-2024 under the
Program \textit{Ayudas de Recualificaci\'on del Sistema Universitario Españ%
ol para 2021-2023 (Convocatoria 2022) - R.D. 289/2021 de 20 de abril (BOE de
4 de junio de 2021)}. This author wish to thank the ICMAT, Universidad de
Alcalá, and the Plan de Recuperación, Transformación y Resiliencia
(NextGenerationEU) of the Spanish Government for their support.\newline

The work of A. Lastra is also partially supported by the project
PID2019-105621GB-I00 of Ministerio de Ciencia e Innovaci\'on, Spain.\newline

The work of V. Soto-Larrosa has been supported by Consejer\'ia de
Econom\'ia, Hacienda y Empleo of the Comunidad de Madrid through ``Programa
Investigo'', funded by the European Union ``NextGenerationEU''.\newline

These authors are members of the research group AnFAO (Cod.: CT-CE2023/876)
of Universidad de Alcal\'a.





\end{document}